\documentclass[a4paper,12pt]{article} 
\usepackage[T1]{fontenc}
\usepackage[latin2]{inputenc}
\usepackage[a4paper,left=2cm,right=2cm,top=2cm,bottom=2cm]{geometry}
\usepackage {times,authblk,url,ntheorem,enumerate} 
\usepackage {amsmath}
\usepackage {amsfonts}
\usepackage {amssymb}
\usepackage {bm}
\usepackage {graphicx}
\usepackage {color}
\usepackage {float}
\usepackage {hyperref}
\usepackage {verbatim}
\usepackage {soul}
\newtheorem{theorem}{Theorem}[section]

\newenvironment{df}{\medskip\par \noindent{\textbf{Definition:}\ }\begin{itshape}}{\end{itshape}\par\medskip}
\newtheorem{lemma}[theorem]{Lemma}

\def\Box{\raisebox{3pt}{\framebox{\hbox to 3pt{\vbox to 3pt{}}}}}
\newenvironment{proof}{\medskip \noindent{\sc Proof:}}{\quad$\Box$\par\medskip} 
\newenvironment{proof2}[1]{\medskip \noindent{\sc Proof of #1:}}{\quad$\Box$\par\medskip} 
\newenvironment{remark}{\medskip\par \noindent{\textbf{Remark:}\ }}{\par \medskip}
\newcounter{question}
\newenvironment{question}{\refstepcounter{question}\medskip\par \noindent{\textbf{Question \thequestion:}\ }}{\par \medskip}

\newtheorem{corollary}[theorem]{Corollary}
\newtheorem{claim}[theorem]{Claim}

\newcommand{\abs}[1]{\left\lvert{#1}\right\rvert}
\newcommand{\floor}[1]{\left\lfloor{#1}\right\rfloor}
\newcommand\sq{\mathbin{\text{\scalebox{.84}{$\square$}}}}

\DeclareMathOperator{\N}{N(\mathcal{B})}

\DeclareMathOperator{\opt}{opt}

\DeclareMathOperator{\diam}{diam}
\DeclareMathOperator{\cov}{Cov}
\DeclareMathOperator{\dist}{d}
\DeclareMathOperator{\reach}{Reach}
\DeclareMathOperator{\exc}{Exc}
\DeclareMathOperator{\V}{V}
\DeclareMathOperator{\TE}{TE}
\DeclareMathOperator{\coop}{Coop}

\DeclareMathOperator{\CE}{CE}
\DeclareMathOperator{\UE}{UE}
\DeclareMathOperator{\DC}{DC}

\DeclareMathOperator{\ef}{ef}

\title{Optimal pebbling number of the square grid }
\author[1,2]{Ervin Gy\H{o}ri\thanks{gyori.ervin@renyi.mta.hu}}
\author[3,4]{Gyula Y. Katona\thanks{kiskat@cs.bme.hu}}
\author[3]{L\'aszl\'o F. Papp\thanks{lazsa@cs.bme.hu}}
\affil[1]{Alfr\'ed R\'enyi Institute of Mathematics, Budapest, Hungary}
\affil[2]{Department of Mathematics, Central European University, Budapest, Hungary}
\affil[3]{Department of Computer Science and
Information Theory, Budapest University of Technology and Economics, Hungary}
\affil[4]{MTA-ELTE Numerical Analysis and Large Networks Research Group, Hungary}

\begin{document}

\maketitle

\begin{abstract} A pebbling move on a graph removes two pebbles from a vertex and adds one pebble
to an adjacent vertex. A vertex is reachable from a pebble distribution if it is possible
to move a pebble to that vertex using pebbling moves. The optimal pebbling number
$\pi_{\opt}$  is the smallest number m needed to guarantee a pebble distribution of m pebbles
from which any vertex is reachable.
The optimal pebbling number of the square grid graph $P_n\square P_m$ was investigated in several papers \cite{ladder,yerger,gridnote}. In this paper, we present a new method using some recent ideas to give a lower bound on $\pi_{\opt}$.
We apply this technique to prove that $\pi_{\opt}(P_n\square P_m)\geq \frac{2}{13}nm$. Our method also gives a new proof for $\pi_{\opt}(P_n)=\pi_{\opt}(C_n)=\left\lceil\frac{2n}{3}\right\rceil$.
\end{abstract}

\section{Introduction}
Graph pebbling is a game on graphs. It was suggested by Saks and Lagarias to solve a number theoretic problem, which was done by Chung \cite{chung}. The main framework is the following: A distribution of pebbles is placed on the vertices of a simple graph. A pebbling move removes two pebbles from a vertex and places one pebble on an adjacent vertex.   The goal is to reach any specified vertex by a sequence of pebbling moves.  This may be viewed as a transportation problem on a graph where the cost of a move is one pebble.  We begin with some notation needed to state our results.

Let $G$ be a simple graph.  We denote the vertex and edge set of $G$ by $V(G)$ and $E(G)$, respectively.     A \emph{pebble distribution} $P$ is a function from $V(G)$ to the nonnegative integers.  We say that $G$ has $P(v)$ pebbles placed at the vertex $v$ under the distribution $P$.  We say that a vertex $v$ is occupied if $P(v)>0$ and unoccupied otherwise.   The \emph{size} of a pebble distribution $P$, denoted $\abs{P}$, is the total number of pebbles placed on the vertices of $G$.

Let $u$ be a vertex with at least two pebbles under $P$, and let $v$ be a neighbor of $u$.  A \emph{pebbling move} from $u$ to $v$ consists of removing two pebbles from $u$ and adding one pebble to $v$.  That is, a pebbling move yields a new pebble distribution $P'$ with $P'(u) = P(u)-2$ and $P'(v) = P(v) + 1$. A sequence of pebbling moves is \emph{executable} under a pebble distribution if we can apply its moves one after the another so that the number of pebbles is nonnegative after each move on any vertex. We call such a sequence a \emph{pebbling sequence}. 
 We say that a vertex $v$ is \emph{$k$-reachable} under the distribution $P$ if we can obtain, after a pebbling sequence, a distribution with at least $k$ pebbles on $v$.  If $k=1$ we say simply that $v$ is reachable under $P$.   More generally, a set of vertices $S$ is $k$-reachable under the distribution $P$ if, after a pebbling sequence, we can obtain a distribution with at least a total of $k$ pebbles on the vertices in $S$. 

A pebble distribution $P$ on $G$ is \emph{solvable} if all vertices of $G$ are reachable under $P$.
A pebble distribution on $G$ is \emph{optimal} if it is solvable and its size is minimal among all of the solvable distributions of $G$.  Note that optimal distributions are usually not unique. 
 
The \emph{optimal pebbling number} of $G$, denoted by $\pi_{\opt}(G)$, is the size of an optimal pebble distribution.   In general, the decision problem for this graph parameter is NP-complete \cite{NPhard}. 

We denote with $P_n$ and $C_n$ the path and cycle on $n$ vertices, respectively.   The \emph{Cartesian product} $G\sq H$ of graphs $G$ and $H$ is defined in the following way: $V(G\sq H)=V(G)\times V(H)$ and $\{(g_1,h_1),(g_2,h_2)\}\in E(G\sq H)$ if and only if $\{g_1,g_2\}\in E(G)$ and $h_1=h_2$ or $\{h_1,h_2\}\in E(H)$ and $g_1=g_2$.

Let $u$ and $v$ be vertices of graph $G$. The distance between $v$ and $u$, namely the number of edges contained in the shortest path between $u$ and $v$, is denoted by $d(v,u)$. The distance $k$ neighborhood of $v$ contains the vertices whose distance from $v$ is exactly $k$. We denote this set with $N_k(v)$. 

The optimal pebbling number is known for several graphs including paths, cycles \cite{ladder,path1,path2}, caterpillars \cite{caterpillar} and $m$-ary trees \cite{m-ary}.  
The optimal pebbling number of grids has also been investigated.   Exact values were proved for $P_n \sq P_2$ \cite{ladder} and $P_n\sq P_3$ \cite{yerger}. The question for bigger grids is still open. The best known upper bound  for the square grid can be found in \cite{gridnote}. Diagonal induced subgraphs of the square grid was studied in \cite{stairs}. 

Instead of the square grid on the plane it is easier to work with the square grid on the torus. As the plane grid is a subgraph of this, any lower bound on the torus grid will also give a lower bound on the plane grid as well. It is well known that the torus grid is a \emph{vertex transitive} graph, i.e.  given any two vertices $v_1$ and $v_2$ of $G$, there is some automorphism
$f:V(G)\rightarrow V(G)$ 
such that
$f(v_{1})=v_{2}$. Some of our statements will be stated for all vertex transitive graphs.

In this paper we present a new method giving a lower bound on the optimal pebbling number. We obtain $\frac{2}{13}V(G)$ as a lower bound for the optimal pebbling number of the square grid, which is better than the previously known bounds.

In Section \ref{sec2} we show that the concept of excess --- introduced in \cite{yerger} ---  can be used to improve the fractional lower bound on the optimal pebbling number. The higher the total excess, the better the obtained bound on the optimal pebbling number is. The problem is that this method is not standalone, because excess can be zero and zero excess does not give us any improvement. Therefore the main objective of the rest of the paper is to give a lower bound on the excess using some other pebbling tools.

In Section \ref{sec3} we study the concept of cooperation. Cooperation is the phenomenon which makes pebbling hard. We show there, that if cooperation can be bounded from above, then we can state a lower bound on the optimal pebbling number. We invent the tool called cooperation excess, which is a mixture of cooperation and excess. In this section we state and prove several small claims which will be required later to prove Lemma \ref{vertex-excess}. This lemma is the essence of our work. It shows that if  the total excess is small, then there is not much cooperation and if cooperation is huge, then the total excess is also large. Therefore in each case one of our two lower bounds works well.

Unfortunately, the proof of Lemma \ref{vertex-excess} is quite complicated. The third part of Section \ref{sec3} and the whole Section \ref{sec4} contain the parts of this proof.
In Section \ref{sec5} we show a general method which can be used to give a lower bound on the optimal pebbling number. This method relies on Lemma \ref{vertex-excess}. Using this method we show that $\pi_{\opt}(P_n\square P_m)\geq \frac{2}{13}nm$. We also present a new proof for  $\pi_{\opt}(P_n)=\pi_{\opt}(C_n)=\left\lceil\frac{2n}{3}\right\rceil$.

\section{Improving the fractional lower bound}
\label{sec2}

The optimal pebbling number problem can be formulated as the following integer programming problem \cite{fractional}, where $v_1, v_2, \ldots, v_n$ are the vertices of the given graph:

\begin{equation*}
P(v_i)+\sum_{x\in N(v_i)}(p_i(x,v_i)-2p_i(v_i,x))\geq 1 \ \forall i\in \{1,2,\ldots n\} 
\end{equation*}
\begin{equation*}
P(v_j)+\sum_{x\in N(v_j)}(p_i(x,v_j)-2p_i(v_j,x))\geq 0 \ \forall i,j\in \{1,2,\ldots n\} 
\end{equation*}
\begin{equation*}
P(v_i)\geq 0 \text{ integer} \ \forall i\in \{1,2,\ldots n\} 
\end{equation*}
\begin{equation*}
p_i(v_j,v_k)\geq 0 \text{ integer} \ \forall i,j,k\in \{1,2,\ldots n\} 
\end{equation*}
\begin{equation*}
\min \sum_{v\in V(G)}P(v)
\end{equation*}

Its fractional relaxation can be solved efficiently, and its solution
is called the \emph{fractional optimal pebbling number}, which gives a
lower bound on the optimal pebbling number. Originally it was defined
in a bit different way, but this is an equivalent definition. You can
find the details of fractional pebbling in \cite{fractional}.


Notice that
some vertices must be $2$-reachable in a solvable
distribution if there is an unoccupied vertex. Optimal distributions
usually contain many unoccupied and several $2$-reachable
vertices. However, in some sense, $2$, $3$, or more reachability
wastes the effect of pebbles. Also $3$-reachability induces larger waste
than $2$-reachability. In order to measure this waste we use the notion called
\emph{excess}, which was introduced in \cite{yerger}.

\begin{df}
Let $\reach(P,v)$ be the greatest integer $k$ such that $v$ is $k$-reachable under $P$. 
The \emph{excess} of $v$ under $P$ is $\reach(P,v)-1$ if $v$ is reachable and zero otherwise. It is denoted by $\exc(P,v)$.
\end{df}

We are interested in the total amount of waste, therefore we define
the notation of \emph{total excess} of $P$, which is
$\TE(P)=\sum_{v\in \V (G)}\exc(P,v)$.

\begin{df}
An \emph{effect} of a pebble placed at $v$ is the following: $\ef(v)=\sum_{i=0}^{\diam(G)}\left(\frac{1}{2}\right)^i|N_i(v)|$.
\end{df}

Herscovici \textit{et al.} proved that the fractional optimal pebbling number of a vertex-transitive graph is $|V(G)|/\ef(v)$, therefore it is a lower bound on the optimal pebbling number. The corollary of the next theorem improves this bound.

\begin{theorem}
If $P$ is a solvable distribution on 
$G$, then
$$\sum_{v\in V(G)}\ef(v)P(v)\geq |V(G)|+\TE(P).$$
\label{excess_dist}
\end{theorem}

\begin{proof}
It is clear that if a vertex $u$ is $k$-reachable under $P$, then it is mandatory that \\$\sum_{v\in V(G)}\left(\frac{1}{2}\right)^{\dist(v,u)}P(v)\geq k$. Summing these inequalities for all the vertices, we have that
$$\sum_{u\in V(G)}\sum_{v\in V(G)}\left(\frac{1}{2}\right)^{\dist(v,u)}P(v)\geq\sum_{u\in V(G)}\reach(P,u).$$
Exchange the summations on the left side and use the fact that $P$ is solvable on the right side, to obtain that
$$\sum_{v\in V(G)}\sum_{u\in V(G)}\left(\frac{1}{2}\right)^{\dist(v,u)}P(v)\geq\sum_{u\in V(G)}(1+\exc(P,u)).$$
Group the elements of the second sum according to the distance $i$ neighborhoods, to acquire that
$$\sum_{v\in V(G)}\sum_{i=0}^{\diam(G)}\left(\frac{1}{2}\right)^i|N_i(v)|P(v)\geq|V(G)|+\TE(P).$$
\end{proof}

\begin{corollary}
If $P$ is a solvable distribution on a vertex-transitive graph $G$, then
$$|P|\geq \frac{|V(G)|+\TE(P)}{\ef(v)}.$$
\end{corollary}

Naturally, this bound is useless without a proper estimate of total excess. To
say something useful about it we look at the optimal pebbling
problem from a different angle.


\section{
Cooperation between distributions}
\label{sec3}
In this section we talk about cooperation, which makes pebbling hard.


\subsection{Pebbling cooperation}
\begin{df}
Let $P$ and $Q$ be  pebble distributions on graph $G$. Now $P+Q$ is the unique pebble distribution on $G$ which satisfies $(P+Q)(v)=P(v)+Q(v)$. $P$ and $Q$ are \emph{disjoint} when no vertex has pebbles under both distributions. 

\end{df}

\begin{df}
The \emph{coverage} of a distribution $P$ is the set of vertices which are reachable under $P$. We denote the size of this set with $\cov(P)$.
\end{df}
A natural idea to find small solvable distributions is finding a distribution with small size and huge coverage and make it solvable by placing some more pebbles.

In the rest of the section we assume that we add disjoint distributions $P$ and $Q$ together. We would like to establish an upper bound using $\cov(P)+\cov(Q)$ on $\cov(P+Q)$ . Similarly, we are interested in some relation between  $\TE(P+Q)$ and $\TE(P)+\TE(Q)$.

\begin{df} A \emph{cooperation vertex} is neither reachable under $P$ nor $Q$, but it is reachable under $P+Q$. We denote the number of such vertices with $\coop(P,Q)$. 
A \emph{double covered} vertex is reachable under both $P$ and $Q$, we denote the size of their set with $\DC(P,Q)$.
\end{df}

The following claim is a trivial consequence of the definitions.

\begin{claim} 
$\cov(P+Q)= \cov (P)+ \cov(Q)+\coop(P,Q)-\DC(P,Q)$.
\label{szetbontas}
\end{claim}

\begin{df}
We say that a distribution $U$ is a \emph{unit}, if only one vertex has pebbles under $U$.
\end{df}
Units are the building blocks of pebble distributions in the following sense: Any distribution $P$ can be written as $\sum_{u|P(u)>0}P_u$, where $P_u$ is a unit having $P(u)$ pebbles at $u$. Units have two main advantages over other distributions. Their coverage and total excess can be easily calculated:

\begin{claim}
Let $U$ be a unit distribution which places pebbles at vertex $u$. Then we have that
\begin{equation*}
\cov(U)=\sum_{i=0}^{\left\lfloor\log_2(P(u))\right\rfloor}|N_i(u)|,
\end{equation*}

\begin{equation*}
\TE(U)=\sum_{i=0}^{\left\lfloor\log_2(P(u))\right\rfloor}|N_i(u)|\left(\left\lfloor \frac{P(u)}{2^i}\right\rfloor-1\right).
\end{equation*}
\label{unit_prop}
\end{claim}

\subsection{Combining cooperation and excess} 

We would like to distinguish the sources of excess. Does it come from $P$ or $Q$ or does it arise from the ``cooperation of $P$ and $Q$''?

\begin{df}
The \emph{unit excess} of $P$, denoted by $\UE(P)$,  is $\sum_{u|P(u)>0}(TE(P_u))$, where $P_u$ is a unit on $u$ containing exactly $P(u)$ pebbles and all of them are placed at $u$.
\end{df} 

\begin{df}
The \emph{cooperation excess} of a vertex $v$ is $\exc(P+Q,v)-(\exc(P,v)+\exc(Q,v))$. If it is positive, then we say that $v$ has cooperation excess.

Similarly, the \emph{cooperation excess} between $P$ and $Q$ is the total excess of $P+Q$ minus the total excesses of $P$ and $Q$. Denote this with $CE(P,Q)$.
\end{df}

We have mentioned previously, that we can split any pebbling
distribution into disjoint unit distributions. If we get $t$ unit
distributions, then the application of Claim \ref{szetbontas} and the
definition of cooperation excess gives the following results.

\begin{claim}\label{excess_bontas}
Let $P$ be a pebble distribution on $G$ and let $\mathcal{D}$ be a disjoint decomposition of $P$ to unit distributions. Denote the elements of $\mathcal{D}$ with $U_1, U_2, \dots, U_t$. Now
\begin{equation}
\TE(P)=\sum_{i=1}^t\TE(U_i)+\sum_{i=1}^t\CE\left(\sum_{k=1}^{i-1}U_k,U_i\right),
\end{equation}
\begin{equation}
\cov(P)=\sum_{i=1}^t\cov(U_i)+\sum_{i=1}^t\left(\coop\left(\sum_{k=1}^{i-1}U_k,U_i\right)-\DC\left(\sum_{k=1}^{i-1}U_k,U_i\right)\right).\label{eq2}
\end{equation}

\end{claim}

Both $\sum_{i=1}^t\TE(U_i)$ and $\sum_{i=1}^t\cov(U_i)$ can be calculated easily. The ``effect'' of cooperation is calculated in the other, more complicated terms. Lemma \ref{vertex-excess} is going to establish a connection between those quantities in a fruitful way.

\subsection{Connection between cooperation and excess}
Now let us  consider an arbitrary graph $G$, and let $\Delta$ be the
maximum degree of $G$.  In the rest of the section we assume that
$Q=U$ is a unit having pebbles only at vertex $u$ and its size is not
zero. Now we state some basic claims about the recently defined
objects.

\begin{claim} Each cooperation vertex $c$ has a neighbor that has
  cooperation excess.
\end{claim}

\begin{proof}
  A cooperation vertex $c$ is not reachable under $P$ or
  $U$. Therefore none of its neighbors is $2$-reachable under these
  distributions. On the other hand, $c$ is reachable under $P+U$,
  hence there is a neighbor $n$ of $c$ which is $2$-reachable under
  this distribution. This means that $n$ has cooperation excess.
\end{proof}

\begin{df}
If a vertex is not a cooperation vertex and it does not have cooperation excess, then we call it \emph{cooperation free}.
\end{df}
This name is a somewhat
misleading, because these vertices can participate in cooperation in a
sophisticated way. For an example see Figure \ref{coopfree}.

\begin{figure}
\centering
\scalebox{0.9}{\input{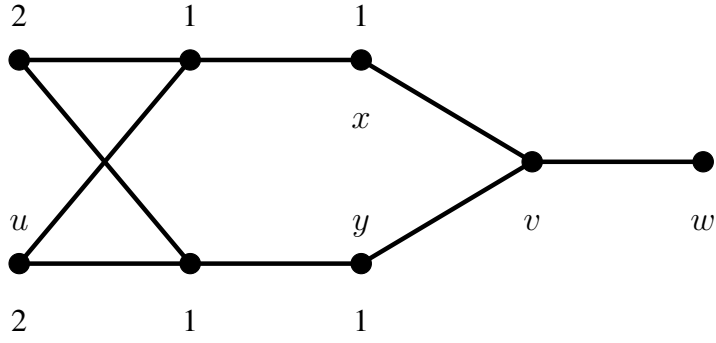}}
\caption{Vertices $x$ and $y$ are both cooperation free, but $v$ has cooperation excess and $w$ is a cooperation vertex.}
\label{coopfree}
\end{figure}

\begin{df} 
Let $\sigma$ be a pebbling sequence. $P_\sigma$ denotes the pebble distribution which is obtained by the application of $\sigma$ to distribution $P$.
A vertex is utilized by a pebbling sequence if there is a move in the sequence which removes or adds a pebble to the vertex.
Let $M(v)$ be the minimal number of cooperation vertices, including $v$ if it is a cooperation vertex, which are utilized by a pebbling sequence $\sigma$ which satisfies that $(P+U)_\sigma(v)\geq 2$. 
If $v$ is not 2-reachable under $P+U$, then we say that $M(v)=\infty$.
\end{df}

An example where $M()$ is shown for each vertex can be seen in Figure \ref{Mvalueexample}.

\begin{figure}
\centering
\scalebox{0.9}{\input{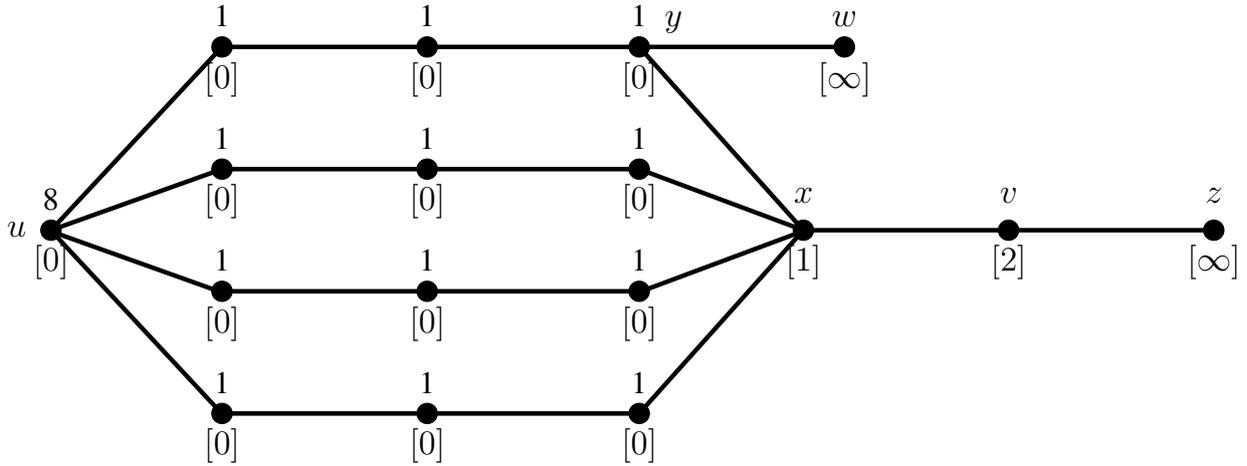}}
\caption{Vertices $w$, $x$, $v$ and $z$ are cooperation vertices. The $M$ values are written in brackets.}
\label{Mvalueexample}
\end{figure}

\begin{claim}
If there is an available pebbling move which removes a pebble from a cooperation vertex $c$, then either two neighbors of $c$, say $e$ and $f$, have cooperation excess at least 1 with $M(e)<M(c)$ and $M(f)<M(c)$ or a neighbor $d$ has cooperation excess  at least 3 with $M(d)<M(c)$.
\label{3cexess}
\end{claim}

\begin{proof}
The condition implies that  $c$ can obtain two pebbles by some pebbling moves under $P+U$. 
Consider a pebbling sequence $\sigma$ which does this by utilizing $M(c)$ cooperation vertices. 
Either $\sigma$ moves the two pebbles to $c$ from two different neighbors $e$ and $f$, or it can move both pebbles from the same neighbor $d$. 
None of the neighbors are $2$-reachable under $P$ or $U$, but $e$, $f$ and $d$ has to be $2$, $2$ and $4$ reachable under $P+U$, respectively. 
This means that $e$ and $f$ have cooperation excess at least 1 and the cooperation excess of $d$ is at least 3. 
Furthermore, $\sigma$ moves two pebbles to $e$ and $f$ or to $d$, then it moves them to $c$ with some more moves. 
This shows that $M(e),M(f),M(d)<M(c)$. 
\end{proof}

\begin{claim}
If the cooperation excess  of a vertex $v$ is at least 3 and one of its neighbors, say $c$, is a cooperation vertex, then there is a vertex $w$ that is adjacent to $v$ and $M(w)\leq M(v)$. 
\label{3asszomszeda}
\end{claim}

\begin{proof}
 Note that $v$ does not have two pebbles under $P+U$, otherwise $c$ can not be
  a cooperation vertex. Vertex $v$ obtains pebbles from its neighbors, so
  one of them, say $w$, can get two pebbles by utilizing at most
  $M(v)$ cooperation vertices. If $v$ is a cooperation vertex\textcolor{red}{,} then a
  pebbling sequence resulting in two pebbles at $v$ utilizes more
  cooperation vertices than the sequence which does not make the final
  move from $w$ to $v$.
\end{proof}

\begin{claim}
If a vertex $v$ has cooperation excess, then it has a neighbor which has cooperation excess or reachable under $P$ or $U$.
\label{twotypeofcoop}
\end{claim}

\begin{proof}
Vertex $v$ gets a pebble under $P+U$, so a neighbor $n$ is $2$-reachable under $P+U$. If $n$ is not 2-reachable under $P$ or $U$, then it has cooperation excess.
\end{proof}

\begin{remark}
In fact, a stronger property holds. If a vertex $v$ gains an extra pebble by cooperation, then it can happen in two ways: A neighbor gained extra pebbles and it passes one of them. Or there are two or more neighbors of $v$ such that each of them can give some pebbles to $v$, but these moves somehow blocks each other. The advantage of the cooperation is that some previously blocked moves can be done simultaneously. This is the way how cooperation free vertices can ``help cooperation''.
\end{remark}

\subsection{Trajectories}
Here we introduce a visualization of pebbling sequences, which is slightly different from the signature digraph used in several pebbling papers (i.e. in \cite{NPhard}). 

\begin{df}
The \emph{trajectory} of a pebbling sequence $\sigma$, denoted by $T(\sigma)$, is a digraph on the vertices of $G$ without parallel edges, where $(u,v)$ is a directed edge if and only if a pebbling move $u\rightarrow v$ is contained in the sequence.

\end{df}
\begin{df}
  The \emph{size of a pebbling sequence} is the total number of moves
  contained in it. We say that $\sigma$ is a \emph{minimal pebbling
    sequence with property $p$} if its size is minimal among all
    pebbling sequences having property $p$.

\end{df}

In the next proof we need a lemma which is frequently used to
solve pebbling problems. It is called No-Cycle Lemma and proved in
several papers \cite{cover_tuza,NPhard,Moews_pebbling}. We state this
lemma in the language of this paper.

\begin{lemma}[No-Cycle \cite{Moews_pebbling}]
Let $P$ be a pebble distribution on graph $G$, and $\sigma$ be a pebbling sequence. There is a subsequence $\delta$ whose trajectory does not contain directed cycles and $P_{\sigma}(v)\leq~P_{\delta}(v)$ for each vertex $v$. 
\end{lemma}

This implies the following corollary:
\begin{corollary}
If $\sigma$ is a minimal pebbling sequence which moves $m$ pebbles to a vertex $v$, then its trajectory is acyclic.
\end{corollary}

\begin{claim}
If $u$ has cooperation excess under $P+U$, where $|U|>0$ , then $u$ is double covered.
\label{Udouble}
\end{claim}

\begin{proof}
The No-Cycle lemma yields that we can move the maximum possible number of pebbles to $u$ without removing a pebble from  $u$. We can move $\reach(U,u)+1$ pebbles to $u$, which means that we move here a pebble of $P$ while we keep the pebbles of $U$, so $u$ is double covered.
\end{proof}

The following definition will be crucial in the
proof.

\begin{df}
We say that a path is a \emph{coopexcess path}, if each inner vertex of the path has cooperation excess. 
\end{df}

\begin{lemma}
Let $v$ be a vertex which is not double covered but it has cooperation excess. 
There is a coopexcess path between $v$ and a double covered vertex or there are at least two cooperation free vertices such that each of them is connected to $v$ by a coopexcess path.
If $v$ is not 2-reachable under both $P$ and $U$, then these paths does not contain a vertex whose $M$ value is higher than $M(v)$. 
\label{Vanjoszomszed}
\end{lemma}

An example for the first case is shown in Figure \ref{Mvalueexample} where $y$ is a double covered vertex and $v,x,y$ is a coopexcess path. The second case can be seen in Figure \ref{coopfree}, where $v,x$ and $v,y$ are coopexcess paths connecting $v$ to cooperation free vertices.

\begin{proof}
Consider a pebbling sequence $\sigma$ moving $\reach(U,v)+\reach(P,v)+1$ pebbles to $v$  utilizing $M(v)$  cooperation vertices. Consider some path in the trajectory of $\sigma$  connecting $u$ to $v$. We can assume that the only sink in the trajectory of $\sigma$ is $v$. A cooperation vertex without cooperation excess can not be the tail of an arc which is contained in the trajectory, therefore each vertex in the trajectory is either cooperation free or it has cooperation excess.

If there is a path between $u$ and $v$ which is contained in the trajectory such that all vertices of this path have cooperation excess, then according to Claim \ref{Udouble} $u$ is double covered and this path is a coopexcess path. If an $u$, $v$ path which is included in the trajectory contains a vertex $d$ which is double covered and each vertex between $d$ and $v$ has cooperation excess, then it is a coopexcess path which we are looking for.
Otherwise, all of the $u,v$ paths which are contained in the trajectory contain cooperation free vertices. 

In each such  path let $w_i$ denote the cooperation free vertex which is the closest vertex to $v$. If  $w_i\neq w_j$ exist, then we have found 2 cooperation free vertices such that each of them is connected to $v$ by a coopexcess path.

In the remaining case there is only one such $w$. 
Either it is a cut vertex in the trajectory or $w=u$. 
Let $T$ be the set of vertices which are included in the trajectory. We divide $T$ to three sets $\mathcal{U}$, $\mathcal{V}$ and $\mathcal{W}$ in the following way:

We remove $w$ from the trajectory obtaining some components, then  we place a vertex $t$ of $T$ to $\mathcal{U}$ if $t$ is in the component containing $u$, similarly we place $t$ to $\mathcal{V}$ if it is in the  component containing $v$ and place the remaining vertices to $\mathcal{W}$. Now we add $w$ to all of these sets. Let $\sigma_u$ be the sequence of pebbling moves containing all moves of $\sigma$ which acts only on the vertices of $\mathcal{U}$. We define $\sigma_w$ and $\sigma_v$ similarly. The sources of the latter two sequences are only $w$ and vertices having pebbles under $P$.

If there is a cooperation free vertex in $\mathcal{V}$ which is not $w$, then the closest one to $v$ is connected to $v$ by a coopexcess path. Hence, assume that all vertices in the trajectory of $\delta_v$ have cooperation excess.

If $w$ is reachable under $U$, then $\sigma_w$ is empty ($w$ is not double covered) and $(P+U)_{\sigma_u}(w)\leq \reach(U,w)$. Since $w$ is cooperation free, we can replace $\sigma_u$ with a pebbling sequence $\delta$ which does not use any pebbles of $P$ and $(P+U)_{\sigma_u}(w)=(P+U)_{\delta}(w)$. Therefore $\delta\sigma_v$ is a pebbling sequence under $P+U$ and  $(P+U)_{\sigma}(v)=(P+U)_{\delta\sigma_v}(v)=\reach(P,v)+\reach(U,v)+1$. $\sigma_v$ must use a pebble of $P$ to do this, otherwise $\delta\sigma_v$ is executable under $U$ which is a contradiction. The trajectory of $\sigma_v$ is connected, therefore there is a vertex which is double covered, furthermore each vertex in this trajectory is connected by a coopexcess path to $v$, so we are done.

If $w$ is not reachable under $U$, then $(P+U)_{\sigma_u\sigma_w}(w)\leq \reach(P,w)$. Thus, there is a minimal pebbling sequence $\delta$ which is executable  under $P$ and $P_\delta(w)=\reach(P+U,w)=\reach(P,w)$. Clearly $\delta\sigma_{v}$ is not executable under $P$ or $(P+U)_{\delta\sigma_v}(v)<(P+U)_{\sigma}(v)$. Both cases require that $\delta$ removes a pebble from a vertex contained in $\mathcal{V}$.

Let $\mathcal{X}\subseteq\mathcal{V}$ be the set of vertices from which $\delta$ removes a pebble. $\delta$ is executable under $P$ so these vertices are $2$-reachable under $P$. Consider the trajectory of $\delta$. If any vertex $x$ from $\mathcal{X}$ is connected in the trajectory with a vertex $y$ contained in $\mathcal{U}$ without pass-through $w$, then each vertex in such a connecting path is 2-reachable under $P$, therefore it is cooperation free or has cooperation excess. So there is either an other cooperation free vertex connected by a coopexcess path to $v$, or there is a coopexcess path between $v$ and $y$ which is connected to $u$ by a path in the trajectory of $\sigma$ which does not contain $w$, so that path has to contain a double covered or a cooperation free vertex, which is not $w$.

The remaining case is when $w$ separates all elements of $\mathcal{X}$ from $\mathcal{U}$ in the trajectory of $\delta$.

Let $\delta_{uw}$ be a maximal subset of $\delta$ which is executable without using the pebbles placed at $\mathcal{X}$, and let $\delta_{v}$ be the remaining subsequence. $\delta_{uw}\sigma_v$ is not executable under $P+U$ or $(P+U)_{\delta_{uw}\sigma_v}(v)<(P+U)_{\sigma}(v)=(P+U)_{\sigma_u\sigma_w\sigma_v}(v)$. Therefore $\sigma_u\sigma_w$ moves more pebbles to $w$ than $\delta_{uw}$, but $\delta_v$ is executable under $P_{\sigma_u\sigma_w}$, thus $P_{\sigma_u\sigma_w\delta_v}(w)>P_\delta(w)$, therefore $w$ has cooperation excess.

To prove the second claim, consider the paths we have found. If 
they were part of the trajectory of $\sigma$, then all of them are 2-reachable under $P+U$, so their $M$ value can not be higher than $M(v)$. Otherwise, the path consists of vertices from the trajectory of $\sigma$ and some others whose $M$ value is zero, since they are 2-reachable under $P$. 
\end{proof}

The following claim is a trivial consequence of the definitions.
\begin{claim}
If $u$ contains at least two pebbles and it is double covered, then one of its neighbors is also double covered.
\end{claim}

In the rest of the section we assume, that $u$ contains at least two pebbles, \textit{i. e.} $|U|\geq 2$. Therefore we can use the previous claim.
\begin{lemma}
Assume that $U$ contains at least two pebbles. Then each double covered vertex $d$ is connected by a coopexcess path to an other double covered vertex or a cooperation free vertex. Furthermore, each vertex of this coopexcess path is $2$-reachable under $U$. 
\label{VanDCvagyures} 
\end{lemma}

\begin{proof}
The previous claim handles the case when $d$ is $u$, since the neighbor is connected to $d=u$. So assume that $d\neq u$.

Since $d$ is  double covered, it is reachable from $U$, so it  is connected to $u$ by a path, whose vertices are $2$-reachable under $U$. Therefore these vertices can not be cooperation vertices. If there is a vertex on this path which does not have cooperation excess, then the vertex closest  to $d$  satisfies the conditions of the second  type. Otherwise, $u$ has cooperation excess which means that it is double covered. 
\end{proof}

We are getting closer to establish a connection between the number of
cooperation vertices and cooperation excess.
\begin{df} We call a subset $\mathcal{Q}$ of $V(G)$  a \emph{$C$-block}, if 
\begin{enumerate}[(1)]
	\item each pair of vertices in $\mathcal{Q}$ is connected by a coopexcess path,
	\item  it contains a vertex having cooperation excess 
\end{enumerate}
and it is maximal with these properties.
\end{df}

Notice that the intersection of two $C$-blocks cannot contain a vertex having cooperation excess.

\begin{lemma}
Each $C$-block either
\begin{enumerate}[(1)]
\setcounter{enumi}{2}
\item\label{cond1} contains at least two double covered vertices, or
\item\label{cond2} contains one double covered vertex and one cooperation free vertex, or
\item\label{cond3} contains at least two cooperation free vertices.
\end{enumerate}
\label{slushlemma}
\end{lemma}

\begin{proof}
Consider an arbitrary element $v$ of $\mathcal{Q}$ which having cooperation excess. 
If the $C$-block does not have a double covered vertex, then Lemma \ref{Vanjoszomszed} guarantees that two cooperation free vertices are connected to $v$ by a coopexcess path, which means that they are contained in $\mathcal{Q}$, so (\ref{cond3}) is satisfied. 

Otherwise $\mathcal{Q}$ contains a double covered vertex.
According to lemma \ref{VanDCvagyures}, either there is an other double covered vertex in $\mathcal{Q}$, or a cooperation free vertex. Thus either (\ref{cond1}) or (\ref{cond2}) is satisfied.  
\end{proof}

Later we generalize the notion of $C$-blocks, so that we keep the properties of \ref{slushlemma}. The following statement will be useful for this.

\begin{lemma}
If a vertex $v$ having cooperation excess is adjacent to a cooperation vertex $c$ such that $M(v)<M(c)$ , then there are vertices $e$ and $f$, such that each of them is either double covered or cooperation free and they are connected to $v$ by coopexcess paths containing only vertices whose $M$ values are smaller than $M(c)$.
\label{coopszomszed}
\end{lemma}

\begin{proof}
Vertex $v$ has a cooperation vertex neighbor, therefore $v$ is not 2-reachable under $P$ or $U$. According to Lemma \ref{Vanjoszomszed} there is a double covered vertex or there are two cooperation free vertices who are connected to $v$ by a coopexcess path containing only vertices whose $M$ values are at most $M(v)<M(c)$. In the latter case we are done. Since the double covered vertex is connected to an other double covered or cooperation free by a coopexcess path containing vertices whose M value is zero, according to Lemma \ref{VanDCvagyures}. The concatenation of these two coopexcess paths fulfills the criteria.
\end{proof}

\section{Connection between total cooperation excess, number of cooperation vertices and maximum degree}
\label{sec4}
In this section we  prove a crucial lemma. Unfortunately, the proof requires quite a lot of effort, including many small claims. 

\begin{lemma}
Let $P$ be an arbitrary pebble distribution on $G$ and $U$ be a unit having at least two pebbles, such that $P$ does not contain a pebble at $u$. Now we have
$$\coop(P,U)-\DC(P,U)\leq (\Delta-2)\CE(P,U).$$ 
\label{vertex-excess} 
\end{lemma}

This lemma gives a connection between the total cooperation, the total
number of double covered vertices and total cooperation excess. The
proof would be relatively easy if the effect of a pebble would appear
close to the location of the pebble. The example on Figure \ref{farcoop}. shows, that
unfortunately this is not always true. 

Another difficulty arises from the fact that a cooperation vertex can have cooperation excess. For such an example see Figure \ref{coopvcoope}. To prove Lemma \ref{vertex-excess} we get rid of such vertices one by one.

\begin{figure}
\centering
\input{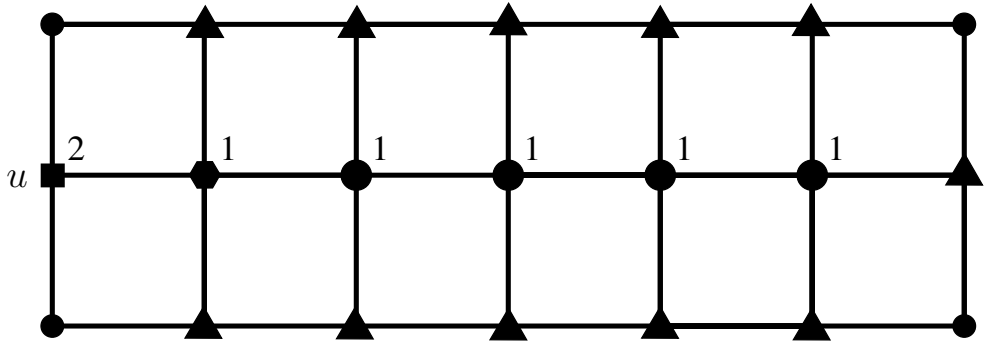}
\caption{The triangles are cooperation vertices. Notice that they can be far from the added unit.}
\label{farcoop}
\end{figure}

\begin{figure}
\centering
\input{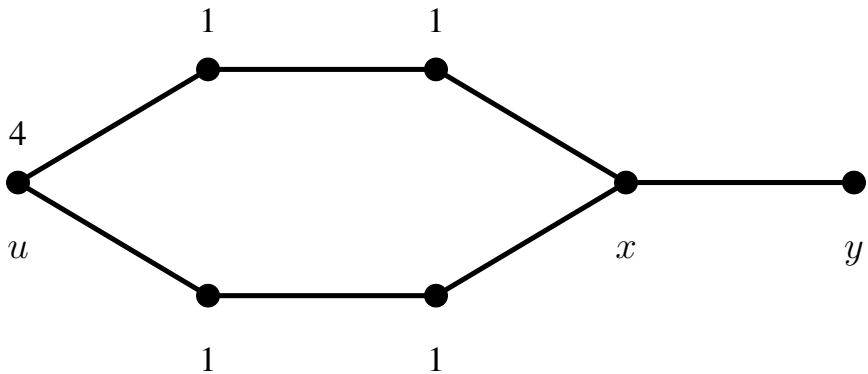}
\caption{Both $x$ and $y$ are cooperation vertices, furthermore $x$ has cooperation excess.}
\label{coopvcoope}
\end{figure}

So to prove the lemma we will change the graph in several steps. In
the new graph it will be easier to isolate these effects.

We introduce a sequence of auxiliary graphs $A_0, A_1,\ldots, A_k$, whose
vertices are labeled with vectors of four coordinates. The first and fourth coordinate is
always an integer, while the other coordinates are binary. We denote
the vertices of these graphs with underlined letters and the $i$th
coordinate of vertex $\underline{b}$ with $b_i$.  We encode the
parameters of the investigated pebbling problem in the auxiliary graph
and in the coordinates in the following way:

$A_0$ is isomorphic to $G$. 
The first coordinate of each vector is the amount of the cooperation
excess of the corresponding vertex. The second coordinate is 1 iff the
corresponding vertex is a cooperation vertex. The third coordinate is 1
when the vertex is double covered. Finally, the last coordinate is
$M(v)$, i.e. the minimum number of cooperation vertices have to be
utilized by a pebbling sequence to obtain 2 pebbles at $v$, where $v$ is the corresponding vertex. So $A_0$
is representation of the original configuration, the labels give the
values of the various quantities that we are interested in. An example can be seen on Fig. \ref{aux}.

The other graphs in the sequence $A_1,\ldots, A_k$ will be obtained
from $A_0$ by applying certain operations recursively, until we
finally obtain $A_k$ with some useful properties. It is
important to note that although the labels of $A_0$ are obtained from
the pebble distribution on $G$, this will not be true any more for
the other auxiliary graphs. We are not trying to change the graph and
the pebble distribution and then obtain the new labels from
these. We just apply the transformation on the abstract, labeled
graphs.

Now we translate the properties of the pebble distribution to properties of $A_0$. 

\begin{figure}
\scalebox{0.9}{\input{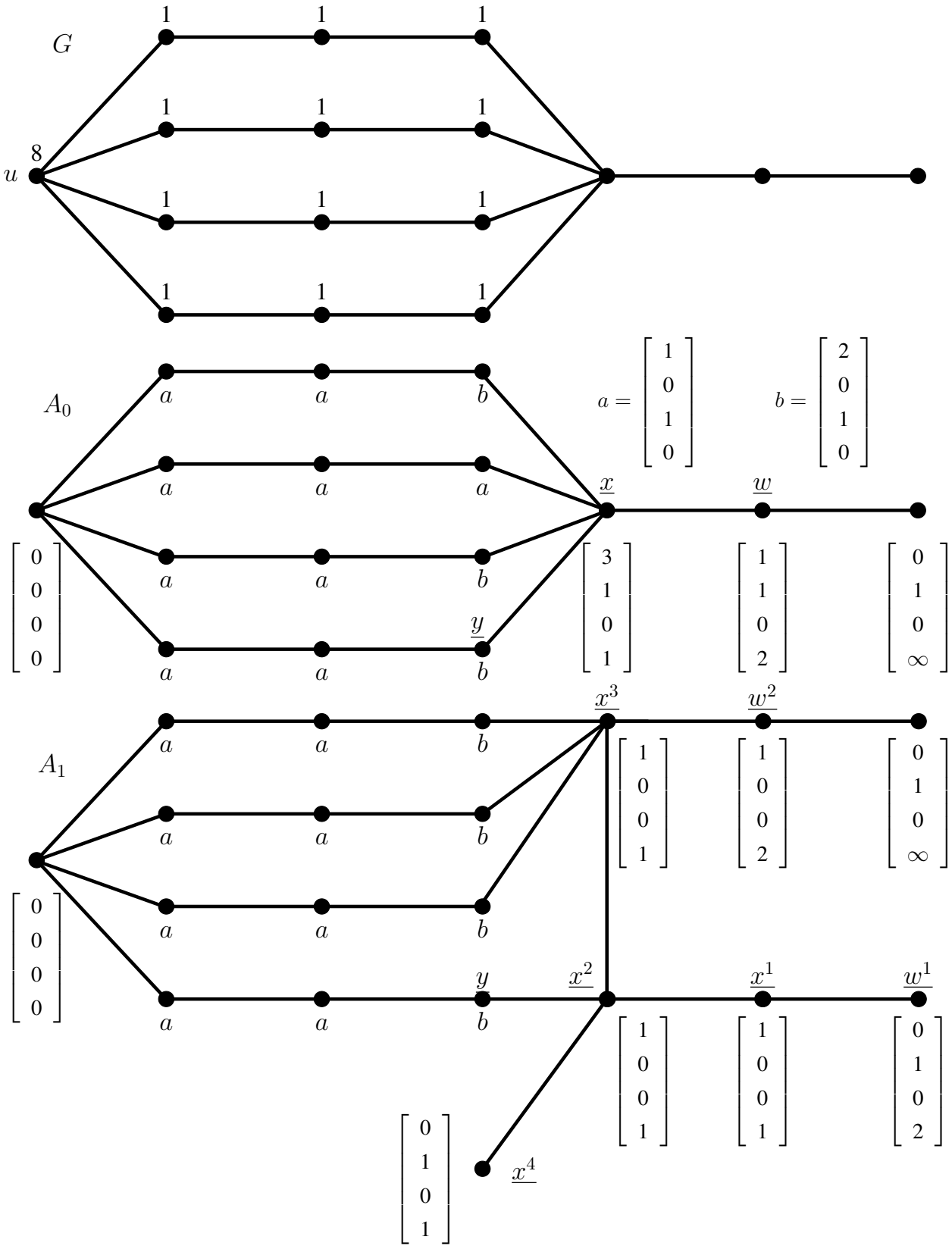}}
\caption{An example graph $G$ with a pebble distribution $P$, a unit $U$ and the corresponding auxiliary graph $A_0$. $A_1$ is obtained from $A_0$ by using the first transformation. Note that $A_1$ does not contain a saturated vertex.}
\label{aux}
\end{figure}

\begin{df}
  We call a path $\mathcal{P}$ in $A$ an \emph{$A$-path}, if each
  inner vertex $\underline{b}$ of $\mathcal{P}$ satisfies $b_1>0$. We
  say that $\mathcal{B}$ is an \emph{$A$-block} iff
\begin{enumerate}[(1)]
\setcounter{enumi}{5}
\item there is a vertex $\underline{b}\in \mathcal{B}$ such that
  $b_1>0$,
\item if $\underline{a},\underline{b}\in \mathcal{B}$, then there is
  an $A$-path which connects them
\end{enumerate}  
and $\mathcal{B}$ is maximal to these properties.
\end{df}
Note that the concept of $A$-path and $A$-block are generalizations of coopexcess path and $C$-block, respectively. In this language, the statement of Lemma \ref{vertex-excess} can be formulated as:

 $$\sum_{\underline{a}\in A}a_2-\sum_{\underline{a}\in A}a_3
\leq\sum_{\underline{a}\in A}a_1(\Delta-2) $$

We state four properties of $A_0$ which will be inherited to later auxiliary graphs. The significant properties are the first and the last. The other two are technical ones which will help the proof of inheritance stated in Claim \ref{oroklodes}.  

\begin{claim}\label{claim4.2}
The following statements hold for $A_0$:
\begin{enumerate}[(1)]
\setcounter{enumi}{7}
\item\label{prop1} If $c_1 c_2>0$ then one of the following two cases hold:
\begin{enumerate}[a)]
\item
 there exist a $\underline{d}$ which is adjacent to $\underline{c}$, $d_1\geq 3$ and $d_4<c_4$, or
\item there are vertices $\underline{e}$ and $\underline{f}$ such that they are neighbors of $\underline{c}$, $e_1$ and $f_1$ are both positive, $e_4<c_4$ and $f_4<c_4$.
\end{enumerate}
\item\label{prop2} If $a_1\geq3$ and $\underline{a}$ has a neighbor $\underline{c}$ such that $c_2=1$, then there is a $\underline{b}$, which is adjacent to $\underline{a}$ and $a_4\geq b_4$. 
\item\label{prop3} Let $\underline{c}$ be a vertex whose first and second coordinates are both positive. If $\underline{a}$ is a neighbor of $\underline{c}$ such that $a_1>0$ and $a_4<c_4$, then there are vertices $\underline{e}$ and $\underline{f}$ such that each of them is connected to $\underline{a}$ by $A$-paths containing only vertices having their fourth coordinate  smaller than $c_4$, and their third coordinate is either $1$ or the first and second coordinates of them are $0$.

\item \label{prop4}
Each $A$-block contains either
\begin{enumerate}[a)]
\item two vertices with third coordinate $1$, or
\item two vertices with first and second coordinates $0$, or
\item one vertex with third coordinate $1$ and one vertex with first and second coordinates $0$.
\end{enumerate}

\end{enumerate}

\label{Aux_prop}
\end{claim}
This claim is equivalent to the following, previously proven, statements with the new notation: Claim \ref{3cexess} $\rightarrow$  (\ref{prop1}), Claim   \ref{3asszomszeda}  $\rightarrow$ (\ref{prop2}), 
Lemma \ref{coopszomszed}  $\rightarrow$ (\ref{prop3}) and Lemma \ref{slushlemma}  $\rightarrow$ (\ref{prop4}).

We will obtain $A_i$ from $A_{i-1}$ by applying one of two
transformations. Then we repeat this until it is possible to apply at
least one of the transformations. Both transformations will preserve
$\sum_{\underline{a}\in A}a_i$, $(i\in \{1,2,3\})$, the fourth
coordinate of each vertex and $\Delta$, the maximum degree in the
graph. The objective of the transformations is to replace vertices
satisfying $a_1 a_2>0$ (i.e. it has cooperation excess and it is a
cooperation vertex) with (one ore more) vertices satisfying $b_1
b_2=0$. From this point, we call these vertices \emph{saturated}
vertices. Both transformations will increase the number of vertices in
the auxiliary graph.

Let $\underline{w}$ be a vertex where $w_1 w_2>0$ such that its fourth
coordinate is maximal among these vertices. By
Claim~\ref{Aux_prop}~(\ref{prop1}) there are two cases.

\textbf{Case 1:} If $\underline{w}$ has a neighbor
$\underline{x}$ such that $x_1\geq 3$ and $w_4>x_4$, then we apply the
following transformation to $A_i$:

\textbf{Transformation 1}
\begin{itemize}

\item Choose a neighbor $\underline{y}$ of $\underline{x}$ such
  that its fourth coordinate is minimal among all neighbors of
  $\underline{x}$.
\item Let $R$ be the set of $\underline{x}$'s neighbors without
  $\underline{y}$ where the product of the first and the second
  coordinate is positive.
\item Delete $x$ and add three vertices $x^1$, $x^2$ and $x^3$,
  such that $x^1_1=x_1^3=1$ and $x^2_1=x_1-2$. $x^1_2=x^2_2=x^3_2=0$,
  $x^1_3=x^3_3=0$ and $x^2_3=x_3$. Connect $\underline{x}^2$ with
  $\underline{y}$, $\underline{x}^1$ and $\underline{x}^3$.
\item Delete each element $\underline{r}$ of $R$ and add two
  vertices $\underline{r}^1$ and $\underline{r}^2$ and set the
  coordinates as: $r_1^1=r^1_3=0$, $r_2^1=1$, $r^1_4=r^2_4=r_4$,
  $r^2_1=r_1$, $r^2_2=0$ and $r^2_3=r_3$. We connect $\underline{r}^1$
  to $\underline{x}^1$ and $\underline{r}^2$ to $\underline{x}^3$ and
  to each original neighbor of $\underline{r}$.
\item We connect the neighbors of $\underline{x}$ which are not included in $R\cup{\underline{y}}$ to $\underline{x}^3$. 
\item Set $x^1_4=x^2_4=x^3_4=x_4$.
\item If $x_2=1$, then  add an extra vertex $x^4$ and connect it only with $\underline{x}^2$. Set its vector to $(0,1,0,x_4)$.

\end{itemize}

\begin{figure}[htb]
\centering
\scalebox{0.75}{\input{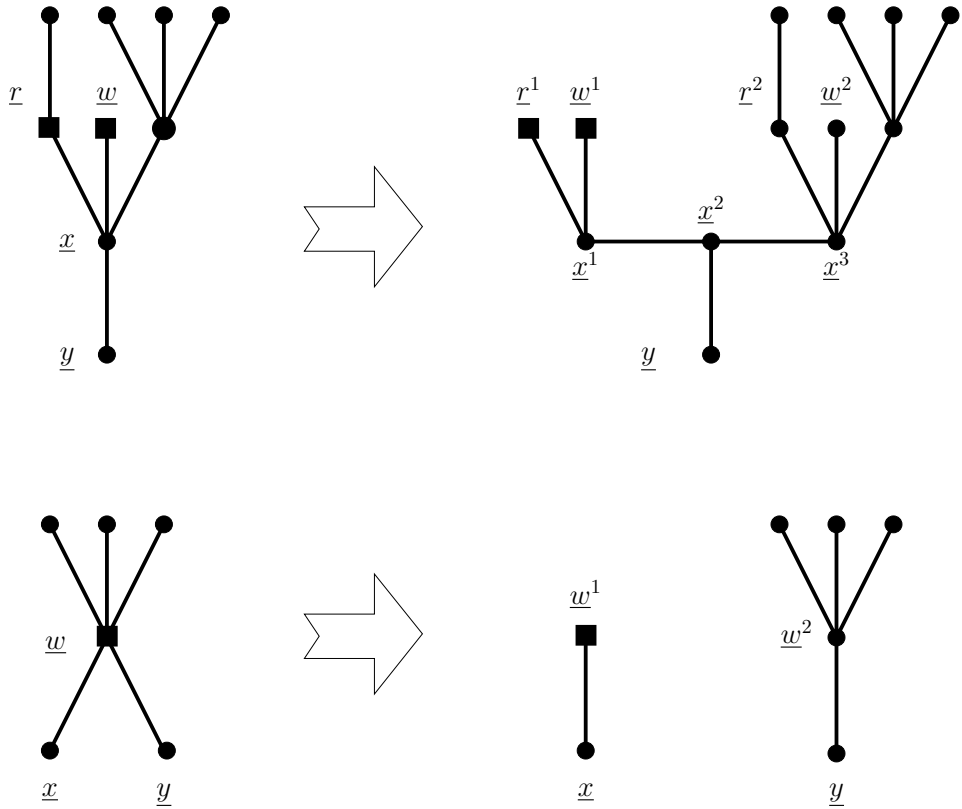}}
\caption{
Vertices denoted by squares are ``cooperation vertices'' so their second coordinates are one. Edges are shown between vertices contained in the same $C$-block. The upper transformation is called first, and the lower one is mentioned as the second transformation. Note that in the upper example $R=\{w,r\}$.}
\label{transformation}
\end{figure}
In other words, this transformation replaces each saturated neighbor
of $\underline{x}$ (excluding a chosen $\underline{y}$) with two
vertices such that one of them is a leaf with zero first coordinate
and the other one is act as the original vertex, but its second
coordinate is zero. To handle the increased degree of $\underline{x}$,
we triple it. Also, if $x$ is saturated then we add the additional
$\underline{x}^4$ vertex. Note that this can be done when $\Delta\geq
4$. If $\Delta\leq 3$, then we have to handle this case in a slightly
different way.

\textbf{Case 2:} If $\underline{w}$ has two neighbors such that their first coordinates are positive and their fourth coordinates are strictly less than $w_4$, then we apply the second transformation:

\textbf{Transformation 2}
\begin{itemize}
\item We choose neighbors $\underline{x}$ and $\underline{y}$ whose fourth coordinate is minimal among all neighbors and $x_4\geq y_4$.
\item We delete $\underline{w}$ and add vertices $\underline{w}^1$ and $\underline{w}^2$. We set the coordinates of these vectors as: $w^1_1=w^1_3=0$, $w^1_2=1$, $w^1_4=w^2_4=w_4$, $w^2_1=w_1$,$ w^2_2=0$ and $w^2_3=w_3$.
\item We connect $\underline{w}^1$ only with $\underline{x}$. In contrast, we connect $\underline{w}^2$ with all neighbors of $\underline{w}$ except $\underline{y}$. 

\end{itemize}

Both transformation can be seen on Figure \ref{transformation}.

\begin{claim}
Both transformations preserve $\sum_{\underline{a}\in A}\underline{a_i}$ $i\in\{1,2,3\}$, 
and $\Delta$ if $\Delta\geq 4$. 
\label{megorzodik}
\end{claim}

\begin{claim}
Both transformations decrease the number of saturated vertices.
\end{claim}

\begin{claim}\label{claim4.5}
If the statements of Claim \ref{Aux_prop} hold for an auxiliary graph, then they hold for the new graph obtained by applying one of the above transformations.
\label{oroklodes}
\end{claim}

\begin{proof}
We say that a vertex $v$ is created by the $i$th transformation if $\underline{v}=\underline{z}^j$ is a vertex of $A_i$ and $z$ is a vertex of $A_{i-1}$.
In this situation we say that $v$ is a descendant of $z$. A vertex is involved in a transformation if either it is created by that or its vector is changed by it.

Notice that the transformations keep the fourth coordinates of the vertices and if two vertices are descendants of the same vertex, then their fourth coordinates are the same.

\textbf{(\ref{prop1}):} 
If $\underline{c}$ is a saturated vertex in $A_i$, then it is not created by the $i$th transformation and it was
saturated in $A_{i-1}$ also. If none of its neighbors were
involved in the last transformation, then the property is clearly
holds. Therefore assume the opposite.
 Assume that $\underline{c}$ had a neighbor $\underline{d}$ in $A_{i-1}$, such that $d_1\geq 3$ and
$d_4<c_4$ in $A_{i-1}$. 

If $\underline{d}$ is contained in $A_i$ also, then the $i$th transformation did not change $d_1$. In that case $\underline{d}$ and $\underline{c}$ are adjacent in
$A_i$ and we are done.

Otherwise, the $i$th transformation created some descendants of $\underline{d}$.

If it was Transformation 1 then a descendant of $\underline{d}$ is connected to $\underline{c}$ and either its first coordinate equals $d_1$ or it is $d_1-2$. In the first case we are done and the latter can happen if and only if $\underline{d}$ acted as $\underline{x}$ in that transformation.
However, this is not possible, because this
would mean that $\underline{c}$ acted as $\underline{y}$ but $y_4\leq
x_4$ by (\ref{prop2}) and the choice of $y$ in Transformation 1, therefore $c_4=y_4\leq x_4=d_4<c_4$ which is a
contradiction. 

The remaining case is that two descendants of $\underline{d}$ are created by Transformation 2. Since $d_4<c_4$, vertex $\underline{c}$ can not be $\underline{x}$ or $\underline{y}$ in Transformation 2, therefore it is adjacent to $\underline{d}^2=\underline{w}^2$ in $A_i$ and the first coordinate of this vertex equals $d_1$.


Now assume that there are neighbors $\underline{e}$ and
$\underline{f}$ in $A_{i-1}$ such that their fourth coordinates are smaller than
$c_4$ and $e_1,f_1>0$.
We may assume that $e_1,f_1<3$, otherwise we
obtain the previous case. Therefore neither $e_1$ nor $f_1$ can act as $\underline{x}$ in Transformation $1$. 

If $\underline{e}$ is contained in $A_i$, then it is still adjacent to $\underline{c}$.
If $\underline{e}$ is replaced with some descendants by the $i$th transformation, 
then one of its descendants keep its first coordinate and that one is connected to $\underline{c}$. Like
in the previous case it cannot happen that $\underline{c}=\underline{y}$ and $\underline{e}=\underline{x}$ in Transformation 1 or $\underline{c}=\underline{x}$ and $\underline{e}=\underline{w}$ in Transformation 2. 
We can state the same for $\underline{f}$.

\textbf{(\ref{prop2}):} 
If $\underline{a}$ is contained in $A_{i-1}$ then $a_1\geq 3$ in $A_{i-1}$ also. Therefore according to ($\ref{prop2}$) there is a $\underline{b}$ which is adjacent to $\underline{a}$ and $b_4\leq a_4$. Either $b_4$ or one of its descendants is adjacent to $\underline{a}$ in $A_i$, therefore we are done.

Otherwise, $\underline{a}$ is a descendant of a vertex $\underline{v}$. $\underline{v}$ has a neighbor $\underline{b}$ whose fourth coordinate is at most $a_4$. There are several cases:

First case: $\underline{v}=\underline{r}$ in Transformation $1$, where $r\in R$. If we remove $\underline{x}$ from the neighborhood of $\underline{v}$ and add $\underline{x^3}$ we obtains the neighborhood of $\underline{a}$. Therefore $\underline{a}$ has a neighbor whose fourth coordinate is not bigger.

Second case: $\underline{v}=\underline{x}$ in Transformation $1$. $\underline{y}$ had the smallest fourth coordinate among the neighbors of $\underline{x}$, thus $y_4\leq x^2_4=a_4$ and $\underline{y}$ and $\underline{a}$ are adjacent in $A_i$.

Third case: $\underline{v}=\underline{w}$ in Transformation $2$. Since $a_1\geq 3$\textcolor{red}{,} $\underline{a}=\underline{w^2}$ and $y_4\leq w_4=w_4^2$ according to (\ref{prop2}).

\textbf{(\ref{prop3}):} Transformation 1 keeps the $A$-paths,
because it keeps connectivity and the first coordinate becomes zero
only at leaves.
Transformation 2 destroys some $A$-paths but all of them contain the saturated vertex which was handled by the transformation and whose fourth coordinate
was at least $c_4$.


\textbf{(\ref{prop4}):} Transformation 1 does not split an
$A$-block, furthermore it keeps the number of vertices whose third coordinate is one and whose first and second coordinate are both zero in each $A$-block.

Transformation 2 either does not split an $A$-block and keeps
the investigated quantities, or it splits an $A$-block to two
$A$-blocks. But (\ref{prop3}) guarantees that both blocks contain
enough vertices whose third coordinate is one or both first and second coordinates are zero.
\end{proof}

\begin{claim}
If Claim \ref{Aux_prop} holds for $A_i$ and there is a saturated vertex, then at least one of the two transformations can be applied to $A_i$.
\end{claim}

The first proposition of Claim \ref{Aux_prop} guarantees this.


\begin{corollary}
There is an $A_k$, such that there is no saturated vertex in $A_k$, furthermore $\sum_{\underline{a}\in A_0}a_i=\sum_{\underline{a}\in A_k}a_i$,
$i\in \{1,2,3\}$.
\end{corollary}

\begin{lemma}
  If $A_k$ does not contain any saturated vertices and Claim
  \ref{Aux_prop} holds, then for each $A$-block $\mathcal{B}$
 $$\sum_{\underline{a}\in \mathcal{B}}a_2-\sum_{\underline{a}\in \mathcal{B}}a_3
\leq\sum_{\underline{a}\in \mathcal{B}}a_1(\Delta-2). $$
\label{izelemma}
\end{lemma}

\begin{df}
  Let $\mathcal{B}$ be an $A$-block in an auxiliary graph. We say that
  a vertex of $\mathcal{B}$ is \emph{inner} vertex if its first
  coordinate is positive otherwise, it is called a \emph{boundary}
  vertex.
\end{df}
\begin{claim} Consider an $A$-block of an auxiliary graph $A$.  Let
  the number of the boundary and the number of inner vertices denoted
  by $b$ and $i$, respectively. If $a_1 a_2=0$ holds for each
  $\underline{a}\in A$ then $b\leq (\Delta-2)i+2$ is satisfied.
\end{claim}

\begin{proof}
  Proof by induction: The base case is an $A$-block with one inner
  vertex. This $A$-block is the closed neighborhood of the only inner
  vertex, therefore the number of boundary vertices is at most
  $\Delta$. Now we assume that for any $i<k$ the inequality is
  true. Let $i=k$. We take a spanning tree of the inner vertices and
  consider a leaf vertex $\underline{l}$. If we set $l_1$ to zero, then $\underline{l}$
  becomes a boundary vertex 
 and at
  most $\Delta-1$ boundary vertices, which are neighbors of $\underline{l}$, are
  dropped from the $A$-block. The number of inner vertices is
  decreased by one, and the number of boundary vertices is decreased
  by at most $\Delta-2$. Using the induction hypothesis the proof is completed.
\end{proof}

\begin{proof2}{Lemma \ref{izelemma}}
  Consider an $A$-block $\mathcal{B}$ and a boundary vertex $v$
  of $\mathcal{B}$. Either $v_2=1$ or $v_1=v_2=0$. Thus we
  have: {$b=\sum_{\underline{a}\in
    \mathcal{B}}a_2+\N$ where $\N$ is the number of vertices in $\mathcal{B}$ whose first two coordinates are zero.} It is also clear that
  $\sum_{\underline{a}\in \mathcal{B}}a_1\geq i$. Combining these
  observations and the previous claim:
$$\sum_{\underline{a}\in \mathcal{B}}a_2+\N=b\leq (\Delta-2)i+2\leq (\Delta-2)\sum_{\underline{a}\in \mathcal{B}}a_1+2. $$
Claim~\ref{claim4.5} implies that  Claim~\ref{Aux_prop}~(\ref{prop4}) holds for $A_k$. This
guarantees that
$\sum_{\underline{a}\in \mathcal{B}} a_3+\N\geq2$. Therefore
$$\sum_{\underline{a}\in \mathcal{B}}a_2+\N\leq (\Delta-2)\sum_{\underline{a}\in \mathcal{B}}a_1+\sum_{\underline{a}\in \mathcal{B}} a_3+\N, $$
$$\sum_{\underline{a}\in \mathcal{B}}a_2-\sum_{\underline{a}\in \mathcal{B}} a_3\leq (\Delta-2)\sum_{\underline{a}\in \mathcal{B}}a_1.$$
\end{proof2}

\begin{proof2}{Lemma \ref{vertex-excess}}
We distinguish three cases depending on $\Delta$.

\textbf{Case 1: $\Delta\geq 4 $}

  Let $A_k$ be the auxiliary graph which we obtained from $A_0$ by
  applying transformations until it does not contain any more saturated
  vertices. The last lemma holds for each $A$-block, therefore:
$$\sum_{\mathcal{B}}\left(\sum_{\underline{a}\in \mathcal{B}}a_2-\sum_{\underline{a}\in \mathcal{B}} a_3\right)\leq \sum_{\mathcal{B}}(\Delta-2)\sum_{\underline{a}\in \mathcal{B}}a_1.$$
Only boundary vertices can be included in multiple blocks, and the
first and third coordinate of a boundary vertex is zero,
thus: $$\sum_{\underline{a}\in A_k}a_2-\sum_{\underline{a}\in A_k}
a_3\leq\sum_{\mathcal{B}}\left(\sum_{\underline{a}\in
    \mathcal{B}}a_2-\sum_{\underline{a}\in \mathcal{B}} a_3\right)\leq
\sum_{\mathcal{B}}(\Delta-2)\sum_{\underline{a}\in \mathcal{B}}a_1=
(\Delta-2)\sum_{\underline{a}\in A_k}a_1.$$
Using Claim \ref{megorzodik} we obtain 
$$\sum_{\underline{a}\in A_0}a_2-\sum_{\underline{a}\in A_0} a_3\leq (\Delta-2)\sum_{\underline{a}\in A_0}a_1.$$

\textbf{Case 2: $\Delta =1,2$}  

If the graph consists of multiple connected components we may restrict our attention to the component containing the unit.  Let $d$ be the number of double covered vertices. We first verify the lemma in the case $\Delta=1$.  In this case, the graph consists of a matching and isolated vertices.  Thus, we must have $\coop(P,P_u) = 0$, and we must show that $CE(P,P_u) \le d$.  If the unit, $u$, is isolated the result is trivial. Suppose the unit is in an edge $\{x,u\}$.  If $P(x)=0$, then $CE(P,P_u) = d = 0$.  If $P(x)=1$, then $CE(P,P_u)= d = 1$.  Suppose that $P(x)=a \ge 2$ and set $\abs{P_u} = b$.  We have $d=2$ and 
\begin{displaymath}
CE(P,P_u) = (a + \floor{b/2} - 1 + b + \floor{a/2} -1) - (a - 1 + \floor{a/2} -1) - (b - 1 + \floor{b/2}-1) = 2.
\end{displaymath} 
This completes the proof in the $\Delta=1$ case.  If $\Delta=2$, then
we may assume that the graph is a path or a cycle.  In this case we
have $\Delta-2 = 0$ so we must show $\coop(P,P_u) \le d$.  However, it
is easy to see that in a path or a cycle every cooperation vertex is
adjacent to a double covered vertex and, moreover, that double covered
vertex is on the path between the cooperation vertex and $u$ (possibly
$u$ itself). It follows that there are at least as many double covered
vertices as cooperation vertices, as desired.

\textbf{Case 3: $\Delta=3$}

We remind the reader that
the problem with the $\Delta=3$ case is that
it is not possible to add $\underline{x}^4$ in transformation 1. 

$\underline{x}^4$ is needed when $\underline{x}$ is saturated in $A_{i-1}$. In that case Transformation 1 handles $\underline{x}$ and substitutes it with unsaturated descendants. If the degree of any of the descendants of $\underline{x}$ is smaller than $\Delta$, then we can make $\underline{x^4}$ adjacent to this vertex and the problem is eliminated. Otherwise $\underline{x}$ has three neighbors: $\underline{y}$, $\underline{v}$ and $\underline{w}$. The one whose fourth coordinate is minimal among them is $\underline{y}$, also $v_4\leq w_4$ and both $\underline{v}$ and $\underline{w}$ are saturated vertices.
 
Now we make one of $\underline{x}$'s descendants saturated. We have to make sure that (\ref{prop1}) holds for this saturated descendant therefore we have to make a few new transformations.

Case 1: (\ref{prop1}) a) holds for $\underline{x}$ in $A_{i-1}$. Consider vertex $\underline{d}$, which is a neighbor of $\underline{x}$ in $A_{i-1}$, $d_4<x_4$ and $d_1\geq 3$. 
If $\underline{d}=\underline{y}$, then we set $x_2^2$ to one, otherwise we set $x_2^3$ to one and the rest of the transformation is similar to Transformation 1.

Case 2: (\ref{prop1}) b) holds for $\underline{x}$ in $A_{i-1}$. Then we apply the following transformation:

\textbf{Transformation 3}

\begin{itemize}

\item Delete $x$ and add three vertices $x^1$, $x^2$ and $x^3$,
  such that $x^1_1=x_1^3=1$ and $x^2_1=x_1-2$. $x^1_2=x^3_2=0$, $x^2_2=1$,
  $x^1_3=x^3_3=0$ and $x^2_3=x_3$. Connect $\underline{x}^2$ with
  $\underline{y}$ and $\underline{x}^3$ and connect $\underline{x^1}$ with $\underline{x^3}$.
\item Delete $\underline{v}$ and $\underline{w}$ and add four vertices $\underline{v}^1$, $\underline{v}^2$, $\underline{w}^1$ and $\underline{w}^2$ and set the
  coordinates as: $v_1^1=v^1_3=0$, $v_2^1=1$, $v^1_4=v^2_4=v_4$,
  $v^2_1=v_1$, $v^2_2=0$ and $v^2_3=v_3$.  We set the coordinates of $w$ in the exact same way.
We connect $\underline{w}^1$ and $\underline{v}^1$ to $\underline{x}^1$. 
We make $\underline{w}^2$ adjacent to the neighbors of $\underline{w}$ and to $\underline{x}^3$. 
We make $\underline{v}^2$ adjacent to the neighbors of $\underline{v}$ and to $\underline{x}^2$.
\end{itemize}

This transformation is shown on Figure \ref{newfig}.

The new transformations are made in such a way that we immediately obtain that (\ref{prop1}) holds for the recently created saturated vertex. The proof of Claim \ref{oroklodes} can be repeated to prove that the statements of Claim \ref{claim4.2} hold after we apply these recently introduced transformations. 
Therefore all the statements following Claim \ref{claim4.2} hold in the $\Delta=3$ case.

\begin{figure}[H]
\centering
\scalebox{0.75}{\input{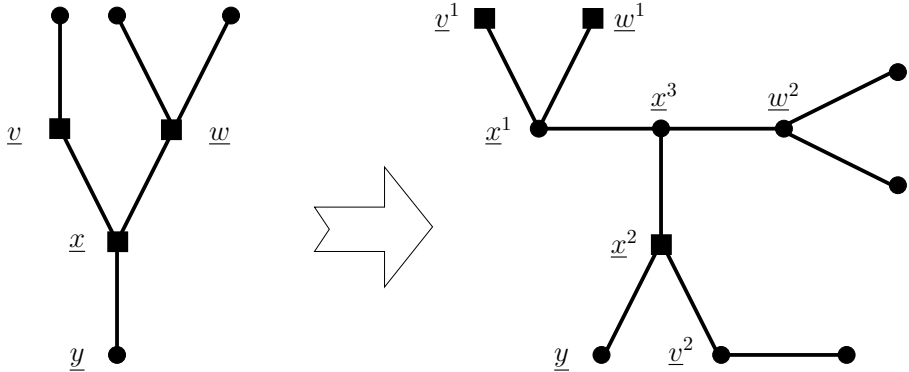}}
\caption{Transformation 3 which is needed when $\Delta=3$.}
\label{newfig}
\end{figure}

\end{proof2}

\section{Lower bound on the optimal pebbling number of vertex
  transitive graphs}
\label{sec5}

\begin{theorem}
  Let $P$ be an arbitrary solvable pebble distribution on $G$ and
  let $\mathcal{D}$ be a disjoint decomposition of $P$ to unit
  distributions. Denote the elements of $\mathcal{D}$ with $U_1, U_2,
  \dots, U_t$, so that $|U_i|\leq |U_{i+1}|$. Now
$$\sum_{i=1}^t\CE\left(\sum_{k=1}^{i-1}U_k,U_i\right)\geq \frac{|V(G)|-\sum_{i=1}^t \cov(U_i)}{\Delta-2}.$$
\end{theorem}

\begin{proof}

  We use Claim \ref{excess_bontas}~(\ref{eq2}) in the following inquality to obtain the second line. To obtain the third line, notice that if
  $|U_i|=1$, then $\coop\left (\sum_{k=1}^{i-1}U_k,U_i\right)=\DC\left
    (\sum_{k=1}^{i-1}U_k,U_i\right)
  =\CE\left(\sum_{k=1}^{i-1}U_k,U_i\right)=0$. Otherwise $|U_i|\geq 2$
  and we can apply Lemma \ref{vertex-excess}.
\begin{equation*}
\begin{split}
|V(G)|&=\cov(P)=\sum_{i=1}^t \left (\cov(U_i)+\coop\left (\sum_{k=1}^{i-1}U_k,U_i\right)-\DC\left (\sum_{k=1}^{i-1}U_k,U_i\right) \right)\\
&= \sum_{i=1}^t \cov(U_i)+ \sum_{i=1}^t\left (  \coop\left (\sum_{k=1}^{i-1}U_k,U_i\right) -\DC\left (\sum_{k=1}^{i-1}U_k,U_i\right)\right)\\
&\leq  \sum_{i=1}^t \cov(U_i)+(\Delta-2)\sum_{i=1}^t\CE\left(\sum_{k=1}^{i-1}U_k,U_i\right)
\end{split}
\end{equation*}

\end{proof}

This result together with the corollary of Theorem \ref{excess_dist} and Claim \ref{excess_bontas} implies the following:

\begin{corollary}
If $P$ is a solvable distribution on a vertex-transitive graph $G$, then 
$$|P|\geq \frac{\frac{\Delta-1}{\Delta-2}|V(G)|+\UE(P)-\frac{1}{\Delta-2}\sum_{i=1}^t \cov(U_i)}{\ef(v)}$$
\label{fokov}
\end{corollary}

This is a tool that helps to prove lower bounds on optimal pebbling
number. Also notice that each element of the formula can be calculated
efficiently.

\subsection{Back to square grids}
We would like to investigate finite square grids. It is easier to
investigate torus graphs instead of square grids, because they are
vertex-transitive. Let $T_{m,n}$ be the torus graph which we obtain if
we glue together the opposite boundaries of $P_{m+1}\square P_{n+1}$. 

Note that $T_{m,n}\cong C_m\square C_n$.

$P_{m}\square P_n$ can be obtained from $T_{m,n}$ by deleting some edges. Edge removal can not decrease the optimal pebbling number, therefore $\pi_{\opt}(T_{m,n})\leq \pi_{\opt}(P_m\square P_n)$. Therefore we work with $T_{m,n}$ in the rest of the section.

The size of the distance $i$ neighborhood in $T_{m,n}$ is at most $4i$. Thus Claim \ref{unit_prop} gives the following estimates on excess and coverage of any unit placed on $T_{m,n}$.

\begin{claim}
Let $U$ be a single unit on $T_{m,n}$. Then:
\begin{equation*}
\cov(U)\leq
\begin{cases}
1 & \text{if } |U|=1 \\
\frac{5}{2}|U| & \text{if } 2\leq |U|\leq 3 \\
\frac{13}{4}|U| & \text{if } |U|\geq 4 \\
\end{cases}
\end{equation*}
\label{cov_becs}
\end{claim}

\begin{claim}
Let $U$ be a single unit on $T_{m,n}$, where $\min(m,n)\geq 5$. We have the following estimate on the ratio of unit excess and the size of the unit:

\begin{equation*}
\exc(U)\geq 
\begin{cases}
0 & \text{if } |U|=1\\
\frac{1}{2}|U| & \text{if } 2\leq  |U|\leq 3\\
\frac{8}{5}|U| & \text{if } 4\leq |U|
\end{cases}
\end{equation*}
\label{exc_becs}
\end{claim}

To obtain these bounds it is enough to check small units and notice that the distance $2$ neighborhood of $u$ contains at least $\frac{8}{5}\abs{U}$ excess when $\abs{U}>4$.

\begin{claim}
\label{ef_becs}
Let $v$ be a vertex of $T_{m,n}$. Then $\ef(v)< 9$.
\end{claim}

A similar result is proven in \cite{yerger} for the square grid. We mimic that calculation.

\begin{proof}
We know that $|N_0(v)|=1$, $|N_1(v)|=4$ and Fig. \ref{neighbors} shows that $|N_i(v)|\leq |N_{i-1}(v)|+4$. Therefore $|N_i(v)|\leq 4i$.

$$\ef(v)=\sum_{i=0}^{\diam(T_{m,n})}\left(\frac{1}{2}\right)^i|N_i(v)| < 1+\sum_{i=1}^{\infty}\left(\frac{1}{2}\right)^i 4i=1+4\sum_{i=1}^{\infty}\sum_{j=i}^{\infty}\left(\frac{1}{2}\right)^j=1+4\sum_{i=1}^{\infty}\left(\frac{1}{2}\right)^{i-1}=9
$$
\end{proof}

\begin{figure}
\centering
\includegraphics[scale=0.7]{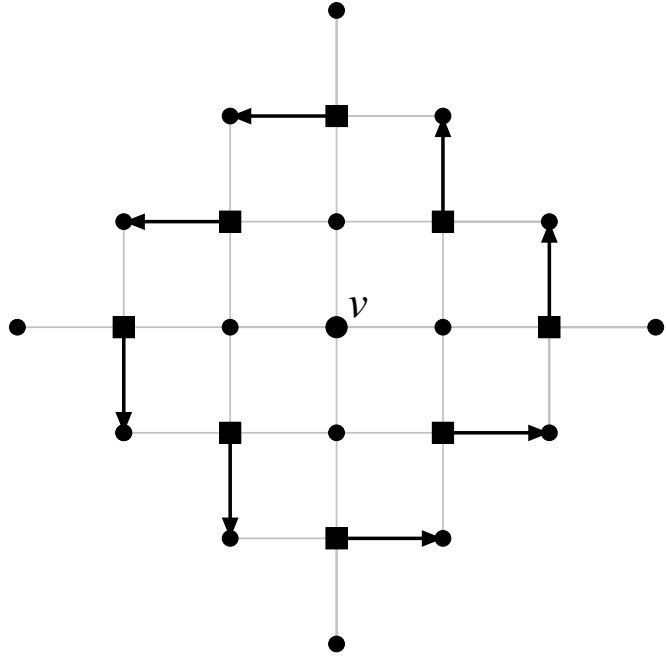}
\caption{A mapping shows that there are at most $4$ more vertices in the distance $3$ neighborhood than in the distance $2$ neighborhood.}
\label{neighbors}
\end{figure}

Now we can obtain our new lower bound on the optimal pebbling number of the square grid: 

\begin{theorem}
The optimal pebbling number of $T_{m,n}$ is at least $\frac{2}{13}nm$, when $m,n\geq 5$.
\end{theorem}

\begin{proof}
\par Let $P$ be an optimal distribution of $T_{m,n}$ and let $\mathcal{D}$ be a disjoint decomposition of $P$ to unit distributions. Denote the elements of $\mathcal{D}$ with $U_1, U_2 \dots U_t$, such that $|U_i|\leq |U_{i+1}|$. Let $\mathcal{D}_{\geq 4}$ be the subset of $\mathcal{D}$ which contains all units whose size is at least four. Furthermore let $\mathcal{D}_{2,3}$ be the set which contains the units whose size is two or three, and  $\mathcal{D}_{1}$ be the set of units whose size is one. Denote the total number of pebbles which are placed on vertices belonging to $\mathcal{D}_{1}$ by $S_1$. Define $S_{2,3}$ and $S_{\geq 4}$ similarly. It is clear that $S_{1}=|P|-S_{2,3}-S_{\geq 4}$.  

We start with Corollary \ref{fokov} and use the estimates of claims \ref{cov_becs}, \ref{exc_becs} and \ref{ef_becs}. 

\begin{equation*}
\begin{split}
  |P|\geq& \frac{\frac{\Delta-1}{\Delta-2}|V(G)|+\UE(P)-\frac{1}{\Delta-2}\sum_{i=1}^t \cov(U_i)}{\ef(v)}\geq\\
  \geq&\frac{\frac{3}{2}nm+\frac{1}{2}S_{2,3}+\frac{8}{5}S_{\geq 4}-\frac{1}{2}\left(\frac{5}{2}S_{2,3}+\frac{13}{4}S_{\geq4}+(|P|-S_{2,3}-S_{\geq4})\right)}{9}=\\
  =&\frac{-\frac{1}{2}|P|+\frac{3}{2}nm-\frac{1}{4}S_{2,3}+\frac{19}{40}S_{\geq
      4}}{9},
\end{split}
\end{equation*}
which implies
\begin{equation*}
|P|\geq \frac{3}{19}nm-\frac{1}{38}S_{2,3}+\frac{1}{20}S_{\geq 4}.
\end{equation*}


Consider the worst case when each of the units contains exactly two or three pebbles:
\begin{equation*}
|P|\geq \frac{3}{19}nm-\frac{1}{38}|P|,
\end{equation*}
thus
\begin{equation*}
|P|\geq \frac{2}{13}nm.
\end{equation*}
\end{proof}

\begin{corollary}
The optimal pebbling number of $P_n\square P_m$ is at least $\frac{2}{13}nm$ when $n,m\geq 5$.
\end{corollary}

\subsection{New proof for the optimal pebbling number of the path and circle}

To illustrate the power of Lemma \ref{vertex-excess} we give a short proof of
the following well known theorem. It was first  proved  in \cite{path1}. Later, essentially different proofs were given in \cite{ladder} and \cite{path2}.

\begin{theorem}
$\pi_{\opt}(P_{3k+r})=\pi_{\opt}(C_{3k+r})=2k+r$ when $0\leq k$, $0\leq r \leq 2$ and $k$, $r$ are integers.
\end{theorem}

The new proof uses Lemma \ref{vertex-excess} when $\Delta=2$. (Note that the proof of this case was short and easy.)

\begin{proof}
It is easy to construct solvable distributions with the desired size, so we prove only the lower bound here.

Let $u$ be a single unit on $P_{3k+r}$ or $C_{3k+r}$. Then:
\begin{equation*}
\frac{\cov(P_u)}{|P_u|}\leq\frac{3}{2} 
\end{equation*}
 Assume that $P$ is a solvable distribution. Now by Lemma \ref{vertex-excess}

\begin{align*}
3k+r=\cov(P)\leq &\sum_{i=1}^t \left (\cov(U_i)+\coop\left (\sum_{k=1}^{i-1}U_k,U_i\right)-\DC\left (\sum_{k=1}^{i-1}U_k,U_i\right) \right)\\ 
\leq &\sum_{i=1}^t \cov(U_i)\leq  \frac{3}{2}|P|.
\end{align*}
So $2k+\frac{2r}{3}\leq |P|$. $|P|$ is integer, therefore this is equivalent to $2k+r\leq |P|$. 
\end{proof}

\section*{Open questions}

\begin{question}
Is there a constant $k$, which does not depend on $n$, such that in an optimal distribution of $P_n\square P_n$ no pebble can be moved to the distance-k neighborhood of its initial location? 
\end{question}

\begin{question}
If the answer for Question 1 is yes, then how small can be $k$?
\end{question}

We think that the answer for Question 1 is yes and we conjecture that $k$ can be 4. Any finite $k$ would improve our lower bound on $\pi_{opt}(P_n\square P_m)$. 

\section*{Acknowledgments}
The authors would like to thank Casey Tompkins for his generous reading of an early draft of this paper, and for his invaluable comments. 

The research of Ervin Gy\H{o}ri and Gyula Y. Katona is partially
supported by National Research, Development and Innovation Office
NKFIH, grant K132696. The research of Gyula Y. Katona and L\'aszl\'o F. Papp is partially
supported by National Research, Development and Innovation Office
NKFIH, grant K124171. 

The research reported in this paper has been supported by the National
Research, Development and Innovation Fund (TUDFO/51757/2019-ITM,
Thematic Excellence Program), and by the Higher Education Excellence
Program of the Ministry of Human Capacities in the frame of  Artificial
Intelligence research area of Budapest University of Technology and
Economics (BME FIKP-MI/SC).

\end{document}